\documentclass[a4paper,11pt,
  leqno
 ]{article}

 \usepackage{comment}
 \usepackage[utf8]{inputenc}
 \usepackage[T1]{fontenc}
 \usepackage{amsthm}
 \usepackage{amssymb}
 \usepackage{mathrsfs}
 \usepackage{amsmath}
 \usepackage{amsfonts}
 \usepackage{mathtools}
 \usepackage{graphicx}
 \usepackage{float}
\usepackage{dsfont}
\usepackage{enumerate}
\usepackage{enumitem}
\usepackage[a4paper]{geometry}
\usepackage{csquotes}
\usepackage{IEEEtrantools}
\usepackage{appendix}
\usepackage{constants}
\usepackage[obeyDraft]{todonotes}
\usepackage{tikz-cd}
\usepackage{enumerate}
\usepackage{mathabx}
\usepackage{nicefrac}
\usepackage{hyperref}
\hypersetup{linktocpage,
  colorlinks   = true, 
  urlcolor     = blue, 
  linkcolor    = blue,
  citecolor   = blue 
}

\newtheorem{The}{Theorem}[section]
\newtheorem{Lemme}[The]{Lemma}

\newtheorem{prop}[The]{Proposition}

\theoremstyle{definition}

\theoremstyle{remark}

\newtheorem{Rk}[The]{Remark}

\numberwithin{equation}{section}
\newcommand{\tend}[2]{\displaystyle\mathop{\longrightarrow}_{#1\rightarrow#2}}

\usepackage{caption}
\title{
\normalsize
\textbf{{
CRITICAL ONE-ARM PROBABILITY FOR THE METRIC \\ GAUSSIAN FREE FIELD IN LOW DIMENSIONS}}}
\author{}
\date{}
\makeatletter
\newcommand{\E}{\mathbb{E}}

\newcommand{\Q}{\mathbb{Q}}

\newcommand{\R}{\mathbb{R}}
\newcommand{\Z}{\mathbb{Z}}

\newcommand{\G}{\mathcal{G}}
\newcommand{\K}{\mathcal{K}}

\renewcommand{\P}{\mathbb{P}}
\newcommand{\eps}{\varepsilon}

\renewcommand{\phi}{\varphi}
\renewcommand{\tilde}{\widetilde}

\renewcommand{\epsilon}{\varepsilon}

\newconstantfamily{c}{symbol=c}

\usepackage{color}
\definecolor{Red}{rgb}{1,0,0}
\definecolor{Blue}{rgb}{0,0,1}
\definecolor{Olive}{rgb}{0.41,0.55,0.13}
\definecolor{Yarok}{rgb}{0,0.5,0}
\definecolor{Green}{rgb}{0,1,0}
\definecolor{MGreen}{rgb}{0,0.8,0}
\definecolor{DGreen}{rgb}{0,0.55,0}
\definecolor{Yellow}{rgb}{1,1,0}
\definecolor{Cyan}{rgb}{0,1,1}
\definecolor{Magenta}{rgb}{1,0,1}
\definecolor{Orange}{rgb}{1,.5,0}
\definecolor{Violet}{rgb}{.5,0,.5}
\definecolor{Purple}{rgb}{.75,0,.25}
\definecolor{Brown}{rgb}{.75,.5,.25}
\definecolor{Grey}{rgb}{.7,.7,.7}
\definecolor{Black}{rgb}{0,0,0}

\usepackage{titletoc}

\dottedcontents{section}[4em]{}{2.9em}{0.7pc}
\dottedcontents{subsection}[0em]{}{3.3em}{1pc}

\usepackage{multicol}
\usepackage{parcolumns}

\usepackage{tabu}
\usepackage{caption}
\begin{document}
\thispagestyle{empty}
\maketitle
\vspace{0.1cm}
\begin{center}
\vspace{-1.9cm}
Alexander Drewitz$^1$, Alexis Pr\'evost$^2$ and Pierre-Fran\c cois Rodriguez$^3$ 

\end{center}
\vspace{0.1cm}
\begin{abstract}
\centering
\begin{minipage}{0.90\textwidth}
We investigate the bond percolation model on transient weighted graphs ${G}$ induced by the excursion sets of the Gaussian free field on the corresponding metric graph. Under the sole assumption that its sign clusters do not percolate, we derive an extension of Lupu's formula for the two-point function at criticality. We then focus on the low-dimensional case $0< \nu < \frac{\alpha}{2}$, where $\alpha$ governs the polynomial volume growth of $G$ and $\nu$ the decay rate of the Green's function on $G$. In particular, this includes the benchmark case ${G}=\Z^3$, for which $\alpha=3$ and $\nu= \alpha-2=1$. We prove under these assumptions that the critical one-arm probability decays with distance $R$ like $R^{-\frac{\nu}{2}}$, up to multiplicative constants.  
\end{minipage}
\end{abstract}

\vspace{5.5cm}
\begin{flushleft}

\noindent\rule{5cm}{0.4pt} \hfill May 2024 \\
\bigskip
\begin{multicols}{2}

$^1$Universit\"at zu K\"oln\\
Department Mathematik/Informatik \\
Weyertal 86--90 \\
50931 K\"oln, Germany. \\
\url{adrewitz@uni-koeln.de}\\[2em]

$^2$University of Geneva\\
Section of Mathematics\\
24, rue du G\'enéral Dufour\\
1211 Genève 4, Suisse.
\\\url{alexis.prevost@unige.ch}\\[2em]

\columnbreak
\thispagestyle{empty}
\bigskip
\medskip
\hfill$^3$Imperial College London\\
\hfill Department of Mathematics\\
\hfill London SW7 2AZ \\
\hfill United Kingdom\\
\hfill \url{p.rodriguez@imperial.ac.uk} 
\end{multicols}
\end{flushleft}

\newpage

\section{Introduction} \label{sec:intro}

The bond percolation problem induced by the excursion sets of the Gaussian free field on metric graphs has recently attracted considerable attention. This model is a variant of a percolation model
introduced by Lebowitz and Saleur \cite{MR865243}, see also \cite{MR914444} for first rigorous results, and more recently re-initiated in \cite{MR3053773}, which concerns percolation of excursion sets of the Gaussian free field on (transient) graphs. The variant is obtained by taking scaling limits along each edge, thereby replacing the original graph by its metric version, following an idea of Lupu \cite{MR3502602}. Both discrete and metric models, along with others (including, notably, the vacant set of random interlacements), are expected to belong to the same universality class \cite{chalhoub2024universality}, characterized by a continuous transition with scaling near the critical point \cite{DrePreRod5}. The continuous structure of the metric graph leads to a degree of integrability of the model. This has prompted significant progress in transient setups, see \cite{DiWi, DrePreRod5, DrePreRod3,MR3502602,werner2020clusters,cai2023onearm,ganguly2024ant,DrePreRod3}, some of which concerns -- remarkably -- the challenging low dimensions, below the mean-field regime.

In this article we focus on the above percolation model, which is defined as follows. Let $\G = (G, \lambda)$ be a weighted graph, with countably infinite vertex set $G$ and symmetric weights $\lambda_{x,y}=\lambda_{y,x} \geq 0$, such that the graph with vertex set $G$ and edge set $\{\{x,y\}: \lambda_{x,y} > 0\}$ is connected and locally finite. The weights give rise to the continuous time Markov chain $X$ on $G$ with generator $
Lf(x)= \frac{1}{\lambda_x} \sum_{y \in G} \lambda_{x,y} (f(y) - f(x)),
$
for suitable $f: G \to \R$, where $\lambda_x = \sum_{y \sim x} \lambda_{x,y}$. This Markov chain is referred to as the \emph{random walk} on $\G$. For $x \in G$ we write $P_x$ for its canonical law when started in $x.$ We assume that $X$ is transient, which is a condition on $\G$. The set $G$ is endowed with a metric $d(\cdot,\cdot)$. For many cases of interest, one can afford to simply choose $d= d_{\text{gr}}$, the graph distance on $\mathcal{G}$, i.e.~$d_{\text{gr}}(x,y)=1$ if and only if $\lambda_{x,y}>0$ (extended to a geodesic distance on $G$); we refer to \cite{DrePreRod2} for an extensive discussion of settings which may require different choices of $d.$ 

In the sequel, the metric graph, or cable system, $\tilde\G$ associated to $\mathcal G$ will play an important role. It is obtained by replacing all edges $\{x,y\}$ by one-dimensional closed intervals of length $ ({2\lambda_{x,y}})^{-1}$, glued through their endpoints; see \cite{MR3502602, DrePreRod3} for precise definitions. The chain $X$ naturally extends to a Markov process on $\tilde\G$ with continuous trajectories. Its canonical law is denoted by $P_x^{\tilde\G}= P_x$ when starting at $ x \in \tilde\G$. 
A set $K\subset \tilde{\mathcal{G}}$ is said to be bounded if $K\cap G$ is a bounded (or equivalently, finite) set.

Attached to the above setup is the random field $\varphi=(\varphi_x)_{x\in \tilde{\G}}$, the mean zero Gaussian free field on $\tilde{\G}$, with canonical law $\P^{\tilde{\G}}=\P$. The percolation problem we are interested in
is obtained by considering excursion sets of $\varphi$ above varying height $a\in \mathbb{R}$. Let $0$ denote an arbitrary point in $\tilde\G$ and consider, for $a\in \R$, \begin{equation}
\label{eq:introKa}
\begin{split}
&{\mathcal{K}}^a \stackrel{\text{def.}}{=} \text{ the connected component of $0$ in $\{ x \in \widetilde{\mathcal{G}}:  \varphi_x \geq a \}$} 
\end{split}
\end{equation}
(with ${\mathcal{K}}^a=\emptyset$ if $\varphi_0 <a$), and the percolation function 
\begin{equation}
\label{eq:intro_theta0}
{\theta}_0(a) \stackrel{\text{def.}} 
=\P({\mathcal{K}}^a \text{ is bounded})
, \quad a \in \R.
\end{equation}
In view of \eqref{eq:intro_theta0}, one defines the critical parameter associated to this percolation model as
\begin{equation}
\label{eq:intro_h_*}
a_* = a_*(\G) =\inf\{ a \in \R: \,  {\theta}_0(a) =1 \}.
\end{equation}
One knows, either by adapting a soft (indirect) argument of \cite{MR914444} or by a (direct) argument involving interlacements and an appropriate isomorphism theorem, that $a_* \geq 0$ for \textit{any} transient $ {\mathcal{G}}$. We will  always assume that
\begin{equation}
\label{eq:critpar0}
 {\theta}_0(0) = 1,
\end{equation}
which is a generic property, see~\eqref{T1_signsat} below. In particular, \eqref{eq:critpar0} implies that
$a_*=0$. The regime $a> 0(=a_*)$ will be referred to as \emph{subcritical} and \eqref{eq:intro_h_*} implies that the probability for $\{ \varphi \geq a\}$ to contain an unbounded cluster (anywhere) vanishes for such $a$. On the other hand, this probability is strictly positive when $a< 0$, which constitutes the \emph{supercritical} regime. 

The fact that \eqref{eq:critpar0} holds follows in practice from the fact that the capacity observable $\mathrm{cap}({\mathcal{K}}^0$ is finite a.s.~on \emph{any} transient graph $\G$ (see \eqref{eq:defcap} for the definition of $\mathrm{cap}(\cdot)$). In fact, rather more is true on graphs that satisfy \eqref{eq:critpar0}. Indeed \cite[Corollary 1.3]{DrePreRod5} implies in this case that 
\begin{equation}\label{eq:cap-tail}
\P\big( \text{cap}({\mathcal{K}}^0) > t \big) \asymp t^{-1/2} \text{ as } t \to \infty.
\end{equation}
The precise asymptotics \eqref{eq:cap-tail} will play an important role in the sequel.
 
We now focus on the one-arm probability for this percolation problem at criticality. For open or closed $U,V \subset \tilde{\G}$, we denote by $\{ U \leftrightarrow  V \}$ the event that $U$ and $V$ are connected by a continuous path in $\tilde{\G}$. Our main interest is in the quantity
\begin{equation}\label{eq:cross-proba}
\psi(R)\stackrel{\text{def.}}{=} \P \big( 0 \leftrightarrow \partial B(0,R) \text{ in }\K^0 \big),
\end{equation}
 where $B(0,R)$ refers to the closed ball of radius $R$ around $0$ in $\tilde{\G}$, see Section~\ref{sec:Lupu} for the precise definition, and $\partial B$ denotes the topological boundary of $B \subset \tilde{\G}$.

The quantitative study of $\psi$ involves further assumptions on $\G=(G,\lambda)$. We will often work under the ellipticity assumption 
\begin{align}
&\label{eq:ellipticity} \tag{$p_0$}
\ \, \lambda_{x,y}/\lambda_x\geq \Cl[c]{c:ellipticity}, \text{ for all }x, y\in{G} \text{ s.t.~$\lambda_{x,y}>0$,} 
\end{align}
for some $\Cr{c:ellipticity} \in (0,1)$ (the terminology \eqref{eq:ellipticity}-condition is borrowed from \cite{MR1853353}). We will also typically require the graph to be $\alpha$-Ahlfors regular, i.e.~there exist a positive exponent $\alpha$ and 
$c,c' \in (0 ,\infty)$ such that the volume growth condition
\begin{equation}
\label{eq:intro_sizeball} \tag{$V_{\alpha}$}
cR^{\alpha}\leq \lambda(B(x,R))\leq c'R^{\alpha}\quad \text{ for all }x\in{G}\text{ and }R\geq1,
\end{equation}
is satisfied. Furthermore we impose that there exist constants $c,c' \in (0,\infty)$ and an exponent $\nu \in (0,\infty)$ such that the Green's function $g$ on $G$ satisfies
\begin{equation}
\label{eq:intro_Green}\tag{$G_{\nu}$}
\begin{split}
&c\leq g(x,x)\leq c' 
\text{ and }  
c d(x,y)^{-\nu}\leq g(x,y)\leq c' d(x,y)^{-\nu} \quad \text{ for all }x \neq y\in{G}.
\end{split}
\end{equation}
Condition~\eqref{eq:intro_Green} alone implies \eqref{eq:critpar0}; see \cite[Lemma 3.4(2) and Corollary 3.3(1)]{DrePreRod3}.
Moreover, when $d=d_{\text{gr}}$ is the graph distance, the above requirements necessarily imply that $(0<) \nu \leq \alpha -2$, see \cite{MR2076770}, and in particular that $\alpha >2$. We will always assume that this is the case. Examples of graphs satisfying the above conditions include the square lattice $\Z^{\alpha}$ for any integer $\alpha\geq3$, for which $\nu=\alpha-2$, as well as various fractal and Cayley graphs \cite{DrePreRod2}. An example with non-integer values of $\alpha,\nu$ is the graphical Sierpinski
carpet in $\alpha$ dimensions; see \cite[Remark~3.10,2)]{DrePreRod2}, for which $1<\nu<2$ when $\alpha=4$. More generally, for any value of $\alpha$ and $\nu$ as before there exists  a graph satisfying the previous conditions \cite{MR2076770}.

We can now formulate our main result. The constants $c,C,$ etc.~below may depend implicitly on $\nu$ and $\alpha$, and the other constants appearing in \eqref{eq:ellipticity}, \eqref{eq:intro_sizeball} and \eqref{eq:intro_Green}.
\begin{The} \label{thm:main}
If $\G$ satisfies \eqref{eq:ellipticity}, \eqref{eq:intro_sizeball} and \eqref{eq:intro_Green} with $0< \nu < \frac{\alpha}{2}$, one has
\begin{equation}
\label{eq:1arm-main}
cR^{-\frac{\nu}{2}} \leq \psi(R) \leq CR^{-\frac{\nu}{2}}, \text{ for all $R \geq 1$.}
\end{equation}
\end{The}

The function $\psi$ has enjoyed a flurry of recent activity. We now briefly survey previous results. The lower bound in \eqref{eq:1arm-main} is known for all $\nu >0$, see Corollary 1.3 and (1.22)--(1.23) in \cite{DrePreRod5}; see also \cite{DiWi} in the case of $\mathbb{Z}^{\alpha}$, for (integer) $\alpha \geq 3$. The article \cite{DrePreRod5} also contains the matching upper bounds when $\nu< 1$, so the crux of Theorem~\ref{thm:main} is the upper bound in \eqref{eq:1arm-main} in the regime when $\nu \geq 1$.
In the special case of $\Z^3$ (with unit weights), corresponding to $\nu =1$, the upper bound $\psi(R) \leq  C \sqrt{\log(R)} R^{-1/2}$ is derived in \cite{DiWi}. For $\Z^\alpha$, $\alpha \geq 4$, the same bound holds without logarithmic pre-factor and is derived in the same reference. As observed in \cite[(1.22)--(1.23)]{DrePreRod5} these bounds hold more generally for $\nu=1$, resp.~$\nu >1$, and are straightforward consequences of the tail asymptotics \eqref{eq:cap-tail} of the cluster capacity observable. The results of \cite{DrePreRod5,DrePreRod8} in fact strongly indicate that the lower bound in \eqref{eq:1arm-main} is sharp below mean-field regime; see \cite[Table~1]{DrePreRod5}. 

More recently, as forecast in \cite{werner2020clusters}, it was proved in \cite{cai2023onearm} that $cR^{-2} \leq \psi(R) \leq CR^{-2}$ on the lattice when $\alpha > 6$ (expectedly the upper-critical dimension based on a formula of Lupu, to which we return below). Even more recently, the upper bounds on $\psi$ were improved in low dimensions. Namely, in \cite{DrePreRod8}, it is shown that $\psi(R) \leq  C {\log \log (R)}^{2/3} R^{-1/2}$ when $\nu=1$ and $\psi(R) \leq  { \log (R)}^{C} R^{-\nu/2}$ for all $1 \leq \nu \leq \frac\alpha{2}$. The result of Theorem~\ref{eq:1arm-main} thus provides up to constant upper bounds in the regime $1 \leq \nu < \frac\alpha{2}$ and in particular for $\nu =1$ (e.g.~on $\Z^3$ with unit weights).

\bigskip
Theorem~\ref{thm:main} is proved by first obtaining a generalization of Lupu's two-point formula \cite{MR3502602}, see \eqref{eq:2point}. This formula is presented in Section~\ref{sec:Lupu}, see Proposition~\ref{P:lupu-gen}, and is of independent interest. It is then used to significantly refine our earlier approach of \cite{DrePreRod8}, which roughly consisted in making rigorous that the cluster $\mathcal{K}^0$ must hit large loops (in the equivalent loop soup picture of the problem) in low dimensions. In fact we will prove the following result, which is effectively a stronger version of Theorem~\ref{thm:main}; we abbreviate $B=B(0,R)$ in the following. 

\begin{The}
\label{eq:explic1armbound} Under the assumptions of Theorem~\ref{thm:main},
for all $R\geq 1$ and $s\leq c$, one has
\begin{equation}
\label{eq:bound1arm+cap}
\begin{split} 
\P\big(0\leftrightarrow \partial B\text{ in }\K^0,\mathrm{cap}(\mathcal{K}^0\cap B) \leq sR^{\nu} \big)\leq R^{-\frac{\nu}{2}}\exp\big(-cs^{-\frac1{\nu}}\big).
\end{split}
\end{equation}
\end{The}

The proof of Theorem~\ref{eq:explic1armbound} appears in Section~\ref{sec:one-arm}. Theorem~\ref{thm:main} is readily obtained from Theorem~\ref{eq:explic1armbound}, as we now explain.

\begin{proof}[Proof of Theorem~\ref{thm:main}]
Fix $s >0$ such that the conclusions of Theorem~\ref{eq:explic1armbound} hold. In view of \eqref{eq:cross-proba}, one has that
$$
\psi(R) \leq \P\big(0\leftrightarrow \partial B\text{ in }\K^0, \, \mathrm{cap}(\mathcal{K}^0\cap B) \leq sR^{\nu} \big) +  \P\big( \mathrm{cap}(\mathcal{K}^0) > sR^{\nu} \big),
$$
and the upper bound in \eqref{eq:1arm-main} follows by combining \eqref{eq:bound1arm+cap} and \eqref{eq:cap-tail}. As explained above, the lower bound in \eqref{eq:1arm-main} is known \cite{DrePreRod5}.
\end{proof}

The results of Theorem~\ref{thm:main} can be fruitfully combined with earlier works \cite{DrePreRod8,prevost2023passage,DrePreRod5,GRS21} to exhibit the scaling behavior of various important observables away from the critical point. This effectively boils down to the fact that one can now afford to set $q(\cdot)=C$ when $\nu<\frac\alpha2$ in various contexts; see for instance \cite[(1.6)]{DrePreRod8} or \cite[(6.1)]{DrePreRod5}. 

We give two examples of this.
First, combining Theorem~\ref{thm:main} and \cite[(1.15)]{DrePreRod8} (see also (1.6) therein), one now knows that for all $a\neq0$, if $\nu<\frac{\alpha}{2}$ and if $d=d_{\mathrm{gr}}$ is the graph distance on $\G$, one has 
    \begin{equation}
    \label{eq:nearcriticalvolume}
        c|a|^{-\frac{2\alpha}{\nu}+2}\leq \E[|\K^a|1\{|\K^a|<\infty\}]\leq C|a|^{-\frac{2\alpha}{\nu}+2},
    \end{equation}
where $|\K^a|$ denotes the cardinality of $\K^a \cap G$. For our second example we specialize to the case of $G=\Z^3$, with unit weights. Then as a consequence of Theorem~\ref{thm:main} and \cite[Corollary 1.3]{DrePreRod8}, one has the following: there exists $c\in (0,\infty)$ and for all $\eta\in{(0,1)}$, there exists $C=C(\eta) \in (0,\infty)$ such that for all $e\in{S^2}$, $a\in{\R}$ with $|a|\leq c$, and all $\lambda\geq C$,     
\begin{equation}
    \label{eq:2pointZ3}
        -\frac{\pi}{6}(1+\eta)\frac{\lambda}{\log(\lambda)}\leq \log\bigg(\frac{\tau^{\mathrm{tr}}_a(0,[\lambda \xi e])}{\tau^{\mathrm{tr}}_0(0,[\lambda \xi e])}\bigg)\leq -\frac{\pi}{6}(1-\eta)\frac{\lambda}{\log(\lambda)},
    \end{equation}
where $\xi=|a|^{-2}$. In depending only on $\lambda$, the bounds \eqref{eq:2pointZ3} give striking evidence of the presumed rotational invariance in the scaling limit at criticality.

\bigskip

\noindent\textit{Note.} During the writing of this article we learned that Zhenhao Cai and Jian Ding obtain the result of Theorem \ref{thm:main} on $\Z^3$ independently, including higher-dimensional Euclidean lattices.

\section{A generalization of Lupu's formula} \label{sec:Lupu}

In this section we derive a key identity that generalizes a result of~\cite{MR3502602}; see Remark~\ref{R:lupu-gen},~\ref{R:lupu} below. We will actually later use this formula on certain subgraphs of $\tilde{\G}$ instead of directly on $\tilde{\G}$, see the proof of Lemma~\ref{lem:newformulacap}, which have a positive killing measure.  Accordingly, throughout this section, we work within the  framework of general transient weighted graphs $\G = (G, \lambda,\kappa)$ with killing measure $\kappa$, thus extending the setup of Section~\ref{sec:intro}. Contrary to the rest of this article, the results of this section actually hold without the conditions \eqref{eq:ellipticity}, \eqref{eq:intro_sizeball} and \eqref{eq:intro_Green}, but in Proposition~\ref{P:lupu-gen} below we will still assume that \eqref{eq:critpar0} is satisfied.

We first introduce some more notation and recall a few known facts, and refer for instance to \cite[Section~2]{DrePreRod3} for more details in this setup. To construct $\tilde{\G}$ on graphs with a positive killing measure, in addition to the closed intervals between edges, we also add for each $x\in{G}$ with $\kappa_x>0$ an interval of length $1/(2\kappa_x)$ starting at $x$, which is closed at $x$ and open on the other side of the interval. Recall that $G$, the vertex set of $\G$, is endowed with a metric $d$.  We write $B(x,R)$ for the subset of the corresponding metric graph $\tilde{\G}$ obtained as the union of all closed intervals between two vertices which are both in the closed ball $\{ y \in {G}: d(x,y) \leq R\}$ of radius $R \geq 0$ around $x \in G$, and all half-open intervals starting at a vertex in the previous ball and with positive killing measure. We write $\partial B(x,R)$ for the set of all the vertices $y$ in $B(x,R)$ such that there is an edge starting in $y$ not included in $B(x,R)$. Recall that $0$ denotes an arbitrary point in $\tilde\G$, which is for instance the origin of the lattice in case $G=\Z^\alpha$, and abbreviate $B(R)=B(0,R)$ and $\partial B(R)=\partial B(0,R)$.

A set $K\subset \tilde{\mathcal{G}}$ is \emph{bounded} if $K\cap G$ is a bounded (or equivalently, finite) set, and \emph{compact} if it is closed, for the natural geodesic distance which assigns length $1$ to each edge, and bounded. We also denote by $\partial K$ the topological boundary of a set $K\subset\tilde{\G}$ for this distance. Note that any half-open interval in $\tilde{\G}$ starting in $x$ with $\kappa_x>0$ is bounded but not compact.

Recall that the random walk $X$ on $\G$  naturally extends to a Markov process on $\tilde\G$ with continuous trajectories. Its canonical law is denoted by $P_x^{\tilde\G}= P_x$ when starting at $ x \in \tilde\G$. 
We denote by $g_U^{\tilde{\G}}= g_U$ for closed $U \subset \tilde{\G}$ the Green's function killed on the set $U$, that is $g_U(x,y)=E_x[\ell_y(H_U)]$, where $H_U=\inf\{t\geq0:\,X_t\in{U}\}$ and $(\ell_y(t))_{y\in{\tilde{\G}},t\geq0}$ is the family of local times associated to $X$. If $x,y\in{G}$ and $U \subset G$, then $g_U(x,y)$ is just equal to $\lambda_y^{-1}$ times the average number of time the discrete time random walk on $G$ started in $x$ visits $y$ before entering $U$. Moreover, we abbreviate $g_U(x)= g_U(x,x) $ for all $x \in \tilde{\G}$, and define the capacity of a finite set $K\subset{G}$ by
\begin{equation}
\label{eq:defcap}
\begin{split} 
\mathrm{cap}(K)=\mathrm{cap}_{\tilde{\mathcal{G}}}(K)\stackrel{\text{def.}}{=}\sum_{x\in{K}}e_K(x),\text{ where }e_K(x)\stackrel{\text{def.}}{=}\lambda_xP_x(\tilde{H}_K=\infty),
\end{split}
\end{equation}
and $\tilde{H}_K$ is the return time to $K$, that is the first time after its first jump that the random walk on $G$ hits $K$. One can extend the definition \eqref{eq:defcap} to closed sets $K\subset\tilde{\G}$ with finitely many connected components, see \cite[(2.20) and (2.27)]{DrePreRod3}, and if $K$ is also compact, then the equilibrium measure $e_K$ is then supported on the finite set $\partial K$. Therefore, in the particular case $\partial K\subset G$, the capacity of $K$ is simply equal to the capacity of $\partial K$ given by \eqref{eq:defcap}. One of the main interests of the capacity is that by \cite[(1.57)]{MR2932978}, which can easily be extended to infinite graphs and to the metric graph, for all compact sets $K\subset\tilde{\G}$ with finitely many connected components, one has
\begin{equation}
\label{eq:hittingcap}
\begin{split} 
P_x(H_K<\infty)=\sum_{y\in{{\partial} K}}g(x,y)e_K(y)\text{ for all }x\in{\tilde{\G}}.
\end{split}
\end{equation}
The Gaussian free field $\varphi=(\varphi_x)_{x\in \tilde{\G}}$ with canonical law $\P^{\tilde{\G}}=\P$ is the mean zero centered Gaussian field with covariance function $g(\cdot,\cdot)$. As to the genericity of the condition \eqref{eq:critpar0}, one knows for instance that 
\begin{equation}
\label{T1_signsat}
\begin{array}{c}
\text{ if }\mathrm{cap}(K)=\infty\text{ for all infinite and connected set }K\subset G,\text{ or if }
\\\text{$\G$ is a vertex-transitive graph (with unit weights), then it satisfies \eqref{eq:critpar0}},
\end{array}
\end{equation}
see Theorem 1.1,(1) and Corollary~1.2 in \cite{DrePreRod3}; see also~\cite[Proposition 8.1]{Pre1} for examples of graphs not verifying \eqref{eq:critpar0}. We can now state the announced formula. Recall that $g_U(x)= g_U(x,x) $ denotes the on-diagonal Green's function killed on $U \subset \tilde{\G}$ and below $U$ is often the origin or its cluster $\mathcal{K}^0$ in $\{\varphi \geq 0\}$ as defined in \eqref{eq:introKa}.

\begin{prop} \label{P:lupu-gen} 
For any weighted graph $\G=(G,\lambda,\kappa)$ satisfying \eqref{eq:critpar0}, all $x \in G$ and $0< t  \leq  g_{\{0\}}(x)$, one has
\begin{equation}
\label{eq:lupu-gen}
\P \big( g_{\{0\}}(x)- g_{\mathcal{K}^0}(x) \geq t\big)= \frac1\pi \arctan\bigg( \frac{g(0,x)}{\sqrt{t g(0)}}\bigg).
\end{equation}
\end{prop}

Before delving into the proof, we make a few comments that shed some light on \eqref{eq:lupu-gen}.

\begin{Rk}
\phantomsection\label{R:lupu-gen}
\begin{enumerate}[label={\arabic*)}]
\item\label{R:lupu-gen-altern} (Alternative formulations). Formula \eqref{eq:lupu-gen} can be equivalently recast as follows. For all $x \in G$ and all $s \in (g(x)-g_{\{0\}}(x), g(x)]$, applying \eqref{eq:lupu-gen} with $t= s- ( g(x)-g_{\{0\}}(x))$, which satisfies $0< t  \leq  g_{\{0\}}(x)$ as required, 
 one finds that
\begin{multline}
\label{eq:lupu-gen-alt}
\P \big( g(x)- g_{\mathcal{K}^0}(x) \geq s\big) \stackrel{\eqref{eq:lupu-gen}}{=}  \frac1\pi \arctan\bigg( \frac{g(0,x)}{\sqrt{t g(0)}}\bigg) \bigg \vert_{t= s- ( g(x)-g_{\{0\}}(x))} \\= \frac1\pi \arcsin\bigg( \frac{g(0,x)}{\sqrt{t g(0) + g(0,x)^2}}\bigg) \bigg \vert_{t= s- ( g(x)-g_{\{0\}}(x))} = \frac1\pi \arcsin\bigg( \frac{g(0,x)}{\sqrt{s g(0)}}\bigg);
\end{multline}
here, the first equality in the second line follows using the trigonometric identity $\arctan(\frac{\alpha}{\beta})= \arcsin(\frac{\alpha}{\sqrt{\beta^2 + \alpha^2}})$ valid for all $\alpha,\beta >0$, and the last equality follows upon observing that $ g(x)-g_{\{0\}}(x) = P_x(H_0<\infty) g(0,x)= \frac{g(0,x)^2}{g(0)}$. Another alternative formulation of \eqref{eq:lupu-gen}, which will be useful in the proof of Lemma~\ref{lem:newformulacap} below, can be obtained from the following alternative description of the quantity appearing on the left-hand side of \eqref{eq:lupu-gen}, which is an easy consequence of the strong Markov property at time $H_{\K^0}$: on the event $\K^0\neq\emptyset$,
\begin{equation}
\label{eq:alt-green-diff}
\begin{split} 
    g_{\{0\}}(x)-g_{\K^0}(x)=E_x\big[g_{\{0\}}(X_{H_{\K^0}},x)1\{H_{\K^0}<H_{\{0\}}\}\big].
\end{split}
\end{equation}
\item \label{R:lupu} (Lupu's formula). Applying \eqref{eq:lupu-gen-alt} with $s= g(x)$ and observing that $g_{\mathcal{K}^0}(x) \geq 0$ with equality if and only if $x \in \mathcal{K}^0$, one immediately deduces that 
\begin{equation}
\label{eq:2point}
\P\big( 0 \stackrel{\{\varphi > 0\}}{\longleftrightarrow} x\big) = \P\big( x \in \mathcal{K}^0\big)  = \frac1\pi \arcsin\bigg( \frac{g(0,x)}{\sqrt{g (x) g(0)}}\bigg),
\end{equation}
thus recovering \cite[Proposition 5.2]{MR3502602}, see also Proposition~2.1 therein; the discrepancy with \cite[display (5.1)]{MR3502602}, where the pre-factor is $\frac2\pi$, is owed to the fact that the latter deals with connection via a loop cluster, and the (independent) cost to have $\text{sign}(\varphi_0)=1$ produces the extra factor $\frac12$.
\end{enumerate}
\end{Rk}

Proposition~\ref{P:lupu-gen} can be proved in several ways. Here we use an approach that combines 
the integrability of the cluster capacity observable and Doob transforms, see \cite{Pre1} in the context of metric graphs, and also \cite{MR3936156}.

\begin{proof}[Proof of Prop.\ \ref{P:lupu-gen}]
As explained in \cite[Remark~2.2]{DrePreRod3}, one can replace any graph with a positive killing measure by a graph with zero killing measure such that the diffusions on the two corresponding metric graphs coincide, and thus the corresponding Gaussian free fields as well. Throughout the proof, we will thus assume for simplicity w.l.o.g.\ that the killing measure on $\G$ is equal to zero. We will prove \eqref{eq:lupu-gen-alt}, from which \eqref{eq:lupu-gen} follows by reverting the arguments of Remark~\ref{R:lupu-gen}.~\ref{R:lupu-gen-altern}.
If $x=0$ then $g_{\{0\}}(x)=0$ and there is nothing to show. We assume henceforth that $x \neq 0$. We write $$\tilde\G_x \stackrel{\text{def.}}{=}\tilde{\G} \setminus \{x\},$$ which is naturally viewed as the metric graph associated to the graph $\G_x$ obtained from $\G=(G,\lambda)$ by removing $x$ from the vertex set $G$, retaining the same weights $\lambda_{y,z}$ for $y,z \in G \setminus \{x\}$ and adding a killing measure $\kappa_y=\lambda_{y,x}$ for each $y \in G$ such that $\lambda_{y,x} > 0$. The half-open interval of length $\frac1{2\kappa_y}=\frac1{2\lambda_{y,x}}$ on the metric graph $\tilde\G_x$ for $y\sim x$ is identified with the closed interval between $x$ and $y$ in $\tilde{\G}$, from which we removed $x$. Note that the diffusion on $\tilde{\G}_x$ then has the same law as the diffusion on $\tilde{\G}$ killed on hitting $x$. By the Markov property for the field one can decompose $\varphi$ under $\P=\P^{\tilde{\G}}$ as
\begin{equation}
\label{eq:markov-x}
\varphi_{\cdot}= \psi_{\cdot}+ \varphi_x h(\cdot)
\end{equation}
(the equality in \eqref{eq:markov-x} defines the field $\psi$), where $\psi_x =0$,  $(\psi_y)_{y \in \tilde{\G}_x}$ has law $\P^{\tilde{\G}_x}$, is independent of $\varphi_x$ and
\begin{equation}
\label{eq:doob-h}
h(y)= P_y(H_x< \infty), \quad y \in \tilde\G.
\end{equation}

The function $h$ in \eqref{eq:doob-h} is harmonic on $\tilde\G_x$ in the sense of \cite[Definition 5.1]{Pre1}, see also (5.2) and (5.3) therein, and we can thus consider the Doob transform $\tilde\G_x^h$
of $\tilde\G_x$ by $h$. By definition, this is the metric graph associated to the graph with same vertex set as $\G_x$ but modified weights $\lambda_{y,z}^h= h(y)h(z) \lambda_{y,z}$ and killing measure $\kappa_y^h=h(y) \kappa_y$. There is a natural isomorphism $\iota: \tilde\G_x \to \tilde\G_x^h $, which acts as identity map on the vertices of $\G_x$ and otherwise stretches the cables `harmonically', see \cite[(5.4)]{Pre1}.

The diffusion on $\tilde\G_x^h$, with law $P_{\iota(y)}^{\tilde\G_x^h}$, $y \in \tilde\G_x$, can be identified with the image under $\iota$ of a time-change of $X$ under $P_{y}^{\tilde\G_x}$, specified as follows. If $(P_t)_{t \geq 0}$ denotes the semigroup of the diffusion $X$ on $\tilde\G_x$ (with law $P_{\cdot}^{\tilde\G_x}$) then the time-changed process in question has semigroup $h^{-1}P_t(h \cdot)$, see \cite[(5.7)]{Pre1}. In more concrete terms, the law $P_{\iota(y)}^{\tilde\G_x^h}$, $y \in \tilde\G_x$, can be viewed up to time-change as the image under $\iota$ of the law of $X$ under $P_y^{\tilde{\G}_x}$ conditionally on $X$ being killed, see \cite[Lemma~5.7]{Pre1}, that is the law of $(X_{t})_{0\leq t < H_x}$ under the measure $P_y^{\tilde \G}(\, \cdot \,  | H_x < \infty)$.

The following result is key. The assumption of connectedness could be weakened but will be sufficient for our purposes.

\begin{Lemme} \label{L:link-cap-g-inv}
For compact connected $K \subset \tilde\G_x $ and $h$ as given by \eqref{eq:doob-h},
\begin{equation}
\label{eq:link-cap-g-inv}
\textnormal{cap}_{\tilde\G_x^h} (\iota(K))= \frac1{g_K(x)} -\frac1{g(x)}.
\end{equation}
\end{Lemme}
The proof of Lemma~\ref{L:link-cap-g-inv} appears below in the present section. 
Observe now that for arbitrary $0< s < g(x)$, the occurrence of the event $\{ s< g(x)- g_{\mathcal{K}^0}(x)  < g(x) \}$ implies that $g_{\mathcal{K}^0}(x) >0$, whence $x \notin \mathcal{K}^0 $, i.e.~$\mathcal{K}^0$ is contained in  $\tilde\G_x$ (and furthermore compact). Thus, Lemma~\ref{L:link-cap-g-inv} applies on this event with the choice $K= \mathcal{K}^0$ (recall that $\mathcal{K}^0$ is bounded $\P$-a.s.\ under \eqref{eq:critpar0}), yielding that
\begin{align}
\label{eq:lupu-gen-pf1}
\begin{split}
\P \big(s< g(x)- g_{\mathcal{K}^0}(x)  < g(x)\big) &= \P \big(\textstyle \frac1{g(x)-s}<  \frac1{g_{\mathcal{K}^0}(x)}  < \infty \big) \\ &\hspace{-.4em}\stackrel{\eqref{eq:link-cap-g-inv}}{=} \P\Big(\textstyle \frac1{g(x)-s} - \frac1{g(x)}<  \textnormal{cap}_{\tilde\G_x^h} (\iota(\mathcal{K}^0))  < \infty \Big).
\end{split}
\end{align}
The probability involving the capacity appearing in \eqref{eq:lupu-gen-pf1} can be explicitly computed, as we now explain. To do so, we use \eqref{eq:markov-x} to view $\mathcal{K}^0 = \mathcal{K}^0(\varphi)= \mathcal{K}^{-\varphi_x h(\cdot)}(\psi)$ and condition on $\varphi_x$. Recall that $\psi$  is independent of $\varphi_x$ under $\P(=\P^{\tilde\G})$ and has law $\P^{\tilde\G_x}$. Now, importantly, by \cite[(5.9)]{Pre1} applied with $\tilde\G_x$ in place of $\tilde\G$, $(\psi_y)_{y\in \tilde\G_x}$ has the same law as $(h(y)\varphi_{\iota(y)})_{y\in \tilde{\G}_x}$ under $\P^{\tilde\G_x^h}$. In particular, this implies that for all $t \in \mathbb{R}$, the set $\iota(\mathcal{K}^{-t h(\cdot)}(\psi))$ has the same law under $\P^{\tilde\G}$ as $\mathcal{K}^{-t}$ under $\P^{\tilde\G_x^h}$. Hence, all in all the probability in the second line of \eqref{eq:lupu-gen-pf1} can be recast as 
\begin{equation}
\label{eq:lupu-gen-pf2}
\frac{1}{\sqrt{2\pi g(x)}}\int_{-\infty}^{\infty} \P^{\tilde\G_x^h}\Big(\textstyle \frac1{g(x)-s} - \frac1{g(x)}<  \textnormal{cap}_{\tilde\G_x^h} (\mathcal{K}^{-t})  < \infty \Big) e^{-\frac{t^2}{2g(x)}} \, {\rm d}t,
\end{equation}
where the integral over $t$ corresponds to averaging over $\varphi_x$. The merit of the rewrite \eqref{eq:lupu-gen-pf2} is that the cluster of $0$ is now at constant height $-t$, cf.~\eqref{eq:lupu-gen-pf1}. Moreover, the following holds.

\begin{Lemme}\label{L:cap-transform}
The distribution of the random variable $\textnormal{cap}_{\tilde\G_x^h} (\mathcal{K}^{-t})1\{ \textnormal{cap}_{\tilde\G_x^h} (\mathcal{K}^{-t}) \in (0,\infty)\}$ under $\P^{\tilde\G_x^h}$ has density given by
$$
\frac{1}{2\pi u \sqrt{ u \frac{ g_{\{x\}}(0)}{h(0)^{2}}  -1}} e^{-\frac{t^2u}{2}} 1_{\big\{ u \geq \frac{h(0)^{2}}{g_{\{x\}}(0)}\big\}}
$$
with respect to Lebesgue measure. 
\end{Lemme}

Lemma~\ref{L:cap-transform} is proved at the end of this section. Feeding the above density into \eqref{eq:lupu-gen-pf2}, applying Fubini and evaluating the Gaussian integral over $t$ using that
$$
\frac{1}{\sqrt{2\pi g(x)}}\int_{-\infty}^{\infty} e^{-\frac{t^2u}{2} - \frac{t^2}{2g(x)} } \, {\rm d}t = \frac{1}{\sqrt{2\pi g(x)}} \cdot \sqrt{\frac{2\pi}{u+ g(x)^{-1}}} = \frac{1}{\sqrt{ug(x)+1}} , \quad u > 0,
$$
it follows in combination with \eqref{eq:lupu-gen-pf1} that for all $s \in (0,g(x))$, 
\begin{equation}
\label{eq:lupu-gen-pf3}
\P \big(s< g(x)- g_{\mathcal{K}^0}(x)  < g(x)\big) = \displaystyle \int_{\alpha}^{\infty} \frac{h(0)}{2\pi u \sqrt{( u { g_{\{x\}}(0)} - {h(0)^{2}})(ug(x)+1)}} \, {\rm d}u,
\end{equation}
where we abbreviated $\alpha = (\frac1{g(x)-s} - \frac1{g(x)}) \vee \frac{h(0)^{2}}{ g_{\{x\}}(0)} $.
The remainder of the proof is computational. Before proceeding, we first collect the following two useful identities, which follow immediately by the (strong) Markov property for the diffusion on $\tilde{\G}$: for all $x \in \tilde\G$, recalling $h$ from \eqref{eq:doob-h} and that $g_U(x)=g_U(x,x)$ for $U \subset \tilde\G$ with $g=g_{\emptyset}$, one readily obtains that
\begin{align}
\label{eq:id1} &g(0,x)= h(0)g(x) \quad \text{and}\\
\label{eq:id2} &g_{\{x\}}(0)= g(0)-h(0)g(x,0)\stackrel{\eqref{eq:id1}}{=} g(0)-h(0)^2g(x).
\end{align}
 We now consider the regime of $s (< g(x))$ such that $\frac1{g(x)-s} - \frac1{g(x)} \geq  \frac{h(0)^{2}}{ g_{\{x\}}(0)}$ (so $\alpha =  \frac1{g(x)-s} - \frac1{g(x)}$ in \eqref{eq:lupu-gen-pf3}), or equivalently that
\begin{equation}
\label{eq:lupu-gen-pf4}
 s \geq \frac{h(0)^2g(x)^2}{g_{\{x\}}(0) + h(0)^2g(x)} \stackrel{\eqref{eq:id1}, \eqref{eq:id2}}{=} \frac{g(0,x)^2}{g(0)}= P_x(H_0<\infty) g(0,x)= g(x)- g_{\{0\}}(x),
\end{equation}
where the penultimate equality is obtained by last-exit decomposition and the last one by the Markov property. Notice that, except for the terminal value $s=g(x)$, which we will deal with separately at the end, the regime of parameters $s (< g(x))$ satisfying \eqref{eq:lupu-gen-pf4} coincides precisely with the one above \eqref{eq:lupu-gen-alt}. For such values of $s$, and with the help of the substitution
$$
u= \frac{g(0,x)^2}{g(x)^2g(0)v^2 - g(0,x)^2g(x)}
$$
one recasts the integral in \eqref{eq:lupu-gen-pf3} with the help of \eqref{eq:id1} and \eqref{eq:id2} to obtain that
\begin{align}
\begin{split}
\label{eq:lupu-gen-pf5}
\P \big(s< g(x)- g_{\mathcal{K}^0}(x)  < g(x)\big) &= \frac1{\pi} \displaystyle \int_{\frac{g(0,x)}{\sqrt{g(0)g(x)}}}^{\frac{g(0,x)}{\sqrt{g(0)s}}} \frac{1}{\sqrt{1-v^2}} \, {\rm d}v \\
&= \frac1{\pi}\bigg( \textstyle \arcsin \Big(\frac{g(0,x)}{\sqrt{g(0)s}} \Big)- \arcsin \Big(\frac{g(0,x)}{\sqrt{g(0)g(x)}} \Big)\bigg)
\end{split}
\end{align}
for all $s \in [\frac{g(0,x)^2}{g(0)}, g(x))$. To deduce \eqref{eq:lupu-gen-alt}, consider now the case $s= \frac{g(0,x)^2}{g(0)}$, for which \eqref{eq:lupu-gen-pf5} yields that
$$
\P \big(\textstyle \frac{g(0,x)^2}{g(0)} < g(x)- g_{\mathcal{K}^0}(x)  < g(x)\big) = 1- \frac12 - \frac1\pi\arcsin \Big(\frac{g(0,x)}{\sqrt{g(0)g(x)}} \Big).
$$
Finally, observe that the random variable $g(x)- g_{\mathcal{K}^0}(x)$ is non-negative and its distribution has an atom at $0$ of weight $\frac12$ (since $g(x)- g_{\mathcal{K}^0}(x)=0$ if and only if $\varphi_0 <0$, which has probability $\frac12$). In combination with the fact that 
$\P\big(0 < g(x)- g_{\mathcal{K}^0}(x) \textstyle \le \frac{g(0,x)^2}{g(0)}\big)=0,
$
it thus follows from the previous display upon rearranging terms that
\begin{equation}
\label{eq:lupu-gen-pf6}
\P \big(g(x)- g_{\mathcal{K}^0}(x)  \geq g(x)\big) \, \Big( = \P \big(g(x)- g_{\mathcal{K}^0}(x)  = g(x)\big) \Big) \, = \frac1\pi\arcsin \Big(\frac{g(0,x)}{\sqrt{g(0)g(x)}} \Big).
\end{equation}
Plugging \eqref{eq:lupu-gen-pf6} into \eqref{eq:lupu-gen-pf5} yields \eqref{eq:lupu-gen-alt}. This completes the proof of Proposition~\ref{P:lupu-gen} under the assumption that both Lemmas~\ref{L:link-cap-g-inv} and~\ref{L:cap-transform} hold.
\end{proof}

We now supply the:

\begin{proof}[Proof of Lemma~\ref{L:link-cap-g-inv}]
We will assume that $\partial K \subset G \setminus\{x\}$, the vertex set of $\G_x$. 
In this case $\partial (\iota(K))$ and $\partial K$ can be naturally identified.
The general case can be reduced to this one by exploiting network equivalence, adding vertices
to $\G_x$ (and $\G_x^h$) corresponding to the boundary points of $K$ (and $\iota(K)$); see \cite[Section~2.2]{DrePreRod3} for further details. Recall the definition of the graph $\G_x^h$ from below \eqref{eq:doob-h}, and notice that by \cite[(5.3)]{Pre1} its weight function $\lambda_{\iota(y)}^h := \kappa_y^h+ \sum_{z \sim y} \lambda_{y,z}^h$ can be recast as $\lambda_{\iota(y)}^h= h(y)^2 \lambda_y$, where $\lambda_y$ is the original weight function on $\G_x$, which coincides with that of $\G$ outside the point $x$. In light of this, and using the concrete characterization of the law $P_{\cdot}^{\tilde\G_x^h}$ described above Lemma~\ref{L:link-cap-g-inv}, one obtains that
\begin{align}
\begin{split}
\label{eq:doob-cap1}
\textnormal{cap}_{\tilde\G_x^h} (\iota(K))&= \sum_{y \in \partial K} \lambda_{\iota(y)}^h P_{\iota(y)}^{\tilde\G_x^h}(\widetilde{H}_K=\infty) = \sum_{y \in \partial K} h(y)^2 \lambda_y P_y^{\tilde\G}(\widetilde{H}_K> H_x | H_x < \infty)\\
&\hspace{-.5em}\stackrel{\eqref{eq:doob-h}}{=} \sum_{y \in \partial K} h(y) \lambda_y P_y^{\tilde\G}(\widetilde{H}_K> H_x) = \sum_{y \in \partial K} h(y) \lambda_x P_x^{\tilde\G}\big(\widetilde{H}_x> H_K, X_{H_K}=y\big),
\end{split}
\end{align}
where the last step uses the reversibility of the walk under $\lambda$. Now, with $(\widehat{X}_n)_{n \geq 0}$ denoting the discrete skeleton of $X$ (under $P_x^{\tilde\G}$) on $G$, which has the law of the random walk on $\G= (G,\lambda)$, decomposing when $\widehat{X}_0=x$ and on the event $\{ H_K<\infty\}$ according to the time of the last visit to $x$ prior to entering $K$, one finds that
\begin{align}\begin{split}
\label{eq:doob-cap2}
&P_x^{\tilde\G}\big(H_K< \infty, X_{H_K}=y\big)\\
&=\sum_{n \geq 0} P_x^{\tilde\G}\Big(\widehat X_k \notin K, k<n, \widehat X_n=x , \widehat X_{n+i} \neq x, 1 \leq i < H_K-n, 
 H_K< \infty, \widehat X_{H_K}=y\Big)\\
 &= g_K(x) \lambda_x P_x^{\tilde\G} \big(H_K < \widetilde{H}_x, X_{H_K}=y \big),
 \end{split}
\end{align}
where the last line follows by an application of the simple Markov property for $\widehat{X}$ at time $n$ (recall that $g_K(x)=g_K(x,x)$). Feeding \eqref{eq:doob-cap2} into \eqref{eq:doob-cap1}, and using that $h(y)= \frac{g(x,y)}{g(x)}$, it follows from a reasoning similar to \eqref{eq:alt-green-diff} that
$$
\textnormal{cap}_{\tilde\G_x^h} (\iota(K))= \sum_{y \in \partial K} \frac{g(x,y)}{g(x) g_K(x)} P_x^{\tilde\G}\big(H_K< \infty, X_{H_K}=y\big)=  \frac{g(x)- g_K(x)}{g(x) g_K(x)}, 
$$
from which \eqref{eq:link-cap-g-inv} is immediate.
\end{proof}

 We now provide the remaining:

\begin{proof}[Proof of Lemma~\ref{L:cap-transform}]
 On the event $\phi_x\geq 0$, the cluster of $0$ in  $\{y\in{\tilde{\G}}:\psi_y\geq 0\}$ is included in $\K^0$ under $\P^{\tilde{\G}}$, and since $\psi$ is independent of $\phi_x$, one deduces from \eqref{eq:critpar0} on $\tilde{\G}$ that  \eqref{eq:critpar0} is also satisfied on $\tilde{\G}_x$ by writing 
 \begin{multline*}
 2^{-1}\P^{\tilde{\G}_x}(\K^0(\varphi) \text{ is bounded})= \P^{\tilde{\G}} (\K^0(\psi) \text{ is bounded}, \, \phi_x\geq 0) \\\geq \P^{\tilde{\G}} (\K^0(\varphi) \text{ is bounded}, \, \phi_x\geq 0) = 2^{-1}\theta_0(0) \stackrel{\eqref{eq:critpar0}}{=}2^{-1},
 \end{multline*}
 where the pemultimate step follows by symmetry, see \cite[Lemma 4.3]{DrePreRod3}. One readily deduces that the free field under $\P^{\tilde\G_x^h}$ has bounded sign clusters a.s, hence $\textnormal{cap}_{\tilde\G_x^h} (\mathcal{K}^{-t})$ has an explicit law given by \cite[Theorem 3.7]{DrePreRod3} (see also \cite[Theorem 5.5]{Pre1}) for all $t\leq 0$. For $t>0$, it follows from \cite[(3.17)]{DrePreRod3} that $\textnormal{cap}_{\tilde\G_x^h} (\mathcal{K}^{-t})$ on the event $\emptyset\neq\mathcal{K}^{-t}$ is compact has the same law as $\textnormal{cap}_{\tilde\G_x^h} (\mathcal{K}^{t})$ on the event $\mathcal{K}^{t}\neq\emptyset$. As we now explain, for $t>0$ and under $\P^{\tilde{\G}_x^h}$, $\mathcal{K}^{-t}$ is actually compact if and only if $\mathrm{cap}(\mathcal{K}^{-t})<\infty$. Indeed any compact set has finite capacity by construction. On the other hand by the isomorphism (Isom') from \cite[p.283]{DrePreRod3}, which is in force in view of Theorem~1.1,2) and (3.14) therein, $\mathcal{K}^{-t}$ has either the same law as $\mathcal{K}^0$ if $\mathcal{K}^0$ does not intersect a trajectory in $\omega_u$, and otherwise stochastically dominates a trajectory in $\omega_u$, where $\omega_u$ is the interlacement process at level $u=t^2/2$ on $\tilde{\G}_x^h$. By the description of the law of $P_{\cdot}^{\tilde{\G}_x^h}$ appearing above the statement of Lemma~\ref{L:link-cap-g-inv}, random interlacement trajectories on $\tilde{\G}_x^h$ a.s.~have images via $\iota$ which contain the interval between $x$ and $y$ in $\tilde{\G}$ for some $y\sim x$, that is the half-open interval starting at $y$ in $\tilde{\G}_x^h$,  and thus have infinite capacity on $\tilde{\G}_x^h$ in view of \cite[(2.32)]{DrePreRod3}. Since $\mathcal{K}^0$ is a.s.~compact on $\tilde{\G}_x^h$, we deduce that if $\mathcal{K}^{-t}$ is non-compact, then it stochastically dominates a trajectory in $\omega_u$, and thus has a.s.~infinite capacity. 

Overall we have that for each $t\in{\R}$, the density with respect to Lebesgue measure of  $\textnormal{cap}_{\tilde\G_x^h} (\mathcal{K}^{-t})$ on the event $0<\textnormal{cap}_{\tilde\G_x^h} (\mathcal{K}^{-t})<\infty$ is given by \cite[(3.7)]{DrePreRod3} on the graph $\tilde{\G}_x^h$. This density depends on $\tilde\G_x^h$ only through $g_{\tilde\G_x^h}(0) =g_{\tilde\G_x^h}(0,0)$, where $g_{\tilde\G_x^h}$ denotes the Green's function of the diffusion on $\tilde\G_x^h$. But as $g_{\tilde\G_x^h}(0)=h(0)^{-2} g_{\{x\}}(0)$ by \cite[(B.1)]{Pre1}, applying \cite[(3.8) and (3.10)]{DrePreRod3} gives the claim.
\end{proof}

\section{Proof of Theorem~\ref{eq:explic1armbound}}\label{sec:one-arm}
Throughout this section, we always tacitly work under the assumption that $\G=(G,\lambda)$ satisfies \eqref{eq:ellipticity}, \eqref{eq:intro_sizeball} and \eqref{eq:intro_Green} with $0< \nu < \frac{\alpha}{2}$; see Section~\ref{sec:intro}. Note that one could also consider the case $\nu=\frac{\alpha}{2}$, but the strategy below would in its present form not lead to any significant improvement over the results of \cite{DrePreRod8}. The symbols $c,C,\dots$ are used for numerical constants (in $(0,\infty)$) that can change from place to place. Their dependence on any parameter other than $\nu$ and $\alpha$, as well as the other constants appearing in our conditions \eqref{eq:ellipticity}, \eqref{eq:intro_sizeball} and \eqref{eq:intro_Green}, will appear explicitly in our notation. In the sequel, we abbreviate $B(L)=B(0,L) \subset \tilde{\G}$ and for all $L\geq1$ introduce the set $\overline{B}(L) \supset B(L)$ as
\begin{equation}
\label{eq:defoverlineball}
\begin{split} 
\overline{B}(L)\stackrel{\text{def.}}{=}
\left\{\begin{array}{c}
x\in{\tilde{\G}}:\text{ any continuous path }\pi=(x_t)_{t\geq0}\subset \tilde{\G}
\\\text{with }x_0=x\text{ and }d(0,x_i)\tend{i}{\infty}\infty\text{ intersects }B(L)
\end{array}\right\}.
\end{split}
\end{equation}
In words, $\overline{B}(L)$ is the set of points which are entirely `surrounded' by $B(L)$. For $R \geq 1$, $\varepsilon > 0$  and $0<2a<b\leq 1$, we let 
\begin{multline}
\label{eq:1-arm-UB1}
t_{R,\varepsilon}^{a,b}\stackrel{\text{def.}}{=}\sup \Big\{ t \geq0:\,\text{ for all }s\leq t,
\\
\P\Big(\mathcal{K}^0 \cap \partial \overline{B}(R) \neq \emptyset, \, \mathrm{cap}(\mathcal{K}^0\cap \mathbb{A}_{R}^{a,b})\leq s((b-a)R)^{\nu}\Big) \leq R^{-\frac{\nu}{2}}e^{-\varepsilon s^{-\frac1\nu}} \Big\},
\end{multline}
where $\mathbb{A}_R^{a,b}$ denotes the `annulus'
\begin{equation}
\label{eq:defann}
\mathbb{A}_R^{a,b}\stackrel{\text{def.}}{=}{B}((1-a)R)\setminus \overline{B}((1-b)R).
\end{equation}
The main ingredient of our proof is the following recursive formula for $ t_{R,\varepsilon}^{a,b}$.

\begin{prop}\label{P:t_R-LB}
There exist $\Cl[c]{c:eps2}>0$ and $c>0 $ such that with $t_R^{\cdot}= t_{R, \Cr{c:eps2}}^{\cdot}$, one has
\begin{equation}
\label{eq:t_R-LB1}
 t_{R}^{a,b} \geq c\Big(a^{\alpha\nu}\wedge\log\Big(\frac1{{t_{R}^{d,e}\vee d}}\Big)^{-\nu}\Big) , \text{ for all $R \geq 1$, $cR^{-1}\leq2a<b<1$ and $2d< e \leq a/4$.}
\end{equation}
\end{prop}

Proposition~\ref{P:t_R-LB} readily implies our main result, as we explain first.
\begin{proof}[Proof of Theorem~\ref{eq:explic1armbound}]
Defining recursively $\log_0(R)=R$ and $\log_{k+1}(R)=\log(\log_k(R))\vee1$, we let
\begin{equation*}
    a_k=\frac{4}{\log_k(R)}\text{ and } b_k=\frac{1}{\log_{k+1}(R)}\text{ for all }k\geq0.
\end{equation*}
We take $a=a_{k+1}$, $b=b_{k+1}$, $d=a_k$ and $e=b_k$ in Proposition~\ref{P:t_R-LB}, and one can verify that the assumptions therein are satisfied for $R\geq C$, which can be assumed w.l.o.g.\ since $\mathrm{cap}(\K^0)\geq c$ whenever $\K^0\neq\emptyset$. Then, abbreviating $u_k=1/t_{R}^{b_k,a_k}$ and recalling that $\alpha\geq1$ (in fact $\alpha > 2$; see below \eqref{eq:intro_Green}), we have by \eqref{eq:t_R-LB1} that 
\begin{equation}\label{eq:indu}
    u_{k+1}\leq \Cl{C:bounduk}\big(\log(u_k)\vee \log_{k+1}(R)\big)^{\alpha\nu}\text{ for all }k\geq0 \text{ such that  }\log_{k+2}(R)\geq 2,
\end{equation}
for some constant $\Cr{C:bounduk}<\infty$. Note that on the event $\K^0\cap \partial \overline{B}(R)\neq\emptyset$, it follows from \cite[(2.8)]{DrePreRod2} that  $\K^0\cap\overline{B}((1-b_0)R)^{\mathsf{c}}\neq\emptyset$  for $R\geq C$, and if $x$ is the closest vertex to $0$ in that set, then $x\in{B((1-a_0)R)}$ for $R$ large enough. Hence $\mathbb{A}_R^{a_0,b_0}\cap\K^0\cap G\neq\emptyset$, and so $\mathrm{cap}(\K^0\cap \mathbb{A}_R^{a_0,b_0})\geq \inf_{x\in{G}}g(x)^{-1}\geq c$ by \eqref{eq:intro_Green}. Therefore, $t_{R}^{a_0,b_0}\geq R^{-\nu}$ for $R\geq C$, which implies  $u_0\leq R^{\nu}$, and one easily deduces from this and \eqref{eq:indu} inductively in $k$ that
\begin{equation*}
\begin{split} 
u_k\leq 2\Cr{C:bounduk}\alpha\nu\log_{k}(R)^{\alpha\nu}\text{ for all $k\geq0$\, such that }\log_k(R)\geq\Cl{C:bounduk2},
\end{split}
\end{equation*}
for some large enough constant $\Cr{C:bounduk2}\in (0,\infty)$. Hence if we denote by $\bar{k}=\bar{k}(R)$ the largest $k$ such that $\log_k(R)\geq \Cr{C:bounduk2}$, we deduce that $u_{\bar{k}}\leq 2\Cr{C:bounduk}\alpha\nu\exp(\Cr{C:bounduk2})$, that is $t_R^{b_{\bar{k}},a_{\bar{k}}}\geq c$. Since $b_{\bar{k}}-a_{\bar{k}}\geq c$ by definition, we obtain \eqref{eq:bound1arm+cap} after a change of variable for $s$, up to replacing $\partial B(R)$ appearing therein by $\partial \overline{B}(R)$. To conclude, it is enough to prove that $\overline{B}(R)\subset B(CR)$ for a large enough constant $C$ and to perform a change of variable for $R$ in \eqref{eq:bound1arm+cap}.  The inclusion can be proved using \eqref{eq:hittingcap} to deduce that for all $x\in{\overline{B}(R)}$, one has
\begin{equation*}
\begin{split} 
1=P_x(H_{B(R)}<\infty)\leq\mathrm{cap}(B(R))\sup_{y\in{B(R)}}g(y,x)\leq CR^{\nu}\Big(\inf_{y\in{B(R)}}d(x,y)\Big)^{-\nu},
\end{split}
\end{equation*}
where the last inequality follows from \eqref{eq:intro_Green} and \cite[(3.11)]{DrePreRod2}.
\end{proof}

We now turn to the proof of Proposition~\ref{P:t_R-LB}, which relies on the formula \eqref{eq:lupu-gen} applied to suitable metric graphs $\tilde{\G}_K \subset \tilde{\G}$ that we now introduce. We consider $K\subset G$ a finite set of vertices and define $\tilde{\G}_K$ as the (only) unbounded connected component of $\tilde{\G}\setminus K$.
Note that $\tilde{\G}_K$ can be identified with the metric graph associated to the graph $\G_K$ having vertex set the unique infinite connected component of $G\setminus K$, the same weights as $\G$ between vertices, and killing measure equal to $\lambda_{y,x}$ for all $y$ in that infinite component which have a neighbor $x$ in $K$, and zero everywhere else. We refer to the beginning of the proof of Proposition~\ref{T1_signsat} for a similar construction. In particular, this identification of $\tilde{\G}_K$ entails that the diffusion on $\tilde{\G}_K$ is well-defined and it is simply the diffusion on $\tilde{\G}$ killed when hitting $K$. Moreover, the graph $\G_K$ fits the setup of Section~\ref{sec:Lupu}, and we can thus define a Gaussian free field on $\tilde{\G}_K$ with canonical law $\P^{\tilde{\G}_K}$. Observe that, in the notation of Section~\ref{sec:Lupu} (see above \eqref{eq:defcap}), this free field has covariance $g^{\tilde{\G}_K}(x,y)= g^{\tilde{\G}}_K(x,y)$, for all $x,y \in \tilde{\G}_K$.

Applying the formula \eqref{eq:lupu-gen}  on the metric graph $\tilde{\G}_K$ yields the following result.
\begin{Lemme}
\label{lem:newformulacap}
There exists $\Cl{c:M}\in{[1,\infty)}$ such that for all values of $R\geq1$, $cR^{-1}\leq a<1/4$,  $K\subset \mathbb{A}_R^{a/2}\stackrel{\text{def.}}{=}{B}((1-a/2)R)$ , $2d<e\leq a/4$, and $t>0$, 
\begin{equation}
\label{eq:newformulacap}
\begin{split} 
\P^{\tilde{\G}_K}\big(\mathrm{cap}(\mathcal{K}^0\cap \mathbb{A}_{R}^{d,e})\geq t((e-d)R)^{\nu}\big)\leq CR^{\frac\nu2}t^{-\frac12}d^{-\frac\nu2}\exp(Ca^{-\alpha})\inf_{x\in{\partial B(\Cr{c:M}R)}}P_0(H_x<H_K).
\end{split}
\end{equation}

\end{Lemme}

The proof of Lemma~\ref{lem:newformulacap} will involve the following random walk estimate, which we show separately first. This relies on the following facts concerning the geometry of $\G=(G,\lambda)$ under our standing assumptions \eqref{eq:ellipticity}, \eqref{eq:intro_sizeball} and \eqref{eq:intro_Green}. Similarly as in \cite[Lemma 6.1]{DrePreRod2}, see also \cite[(2.2)]{DrePreRod8}, we introduce under the above assumptions on $\mathcal{G}$ the approximate renormalized lattice $\Lambda(L)$ for $L \geq 1$ having the following properties. There  exists a constant $\Cl{CLambda}\in (0,\infty)$ such that for all $x\in{G}$ and $L,N\geq1$, 
\begin{equation}
\label{eq:defLambda}
\begin{cases}
  &  \bigcup_{y\in{\Lambda(L)}}B(y,L)=G, \\
  &  \text{the balls }B(y,\tfrac{L}{2}),\ y\in{\Lambda(L)},\text{ are disjoint,}  \\
  &  |\Lambda(L)\cap B(x,LN)|\leq \Cr{CLambda}N^{\alpha}. 
     \end{cases}
\end{equation}
We will use the lattices $\Lambda(\cdot)$ several times in the sequel. We further say that $\pi=(x_i)_{i\leq L}$ is a path in $\Lambda(L)$ if $x_i\in{\Lambda(L)}$ for all $i\leq L$, and for each  $i<L$, there exists $x\in{B(x_i,L)\cap G}$ and $y\in{B(x_{i+1},L)\cap G}$ such that $x$ and $y$ are neighbors in $\G$.  

\begin{Lemme}\label{L:hittingproba}
 For all $R\geq1$, $a \geq c R^{-1}$ and $2d<e\leq a/4$, one has
    \begin{equation}
\label{eq:hittingprobaxz}
\begin{split} 
P_y\big(H_x<H_{\mathbb{A}_R^{a/2}}\big)\geq c\exp(-Ca^{-\alpha})R^{-\nu}, \text{ for all $y\in{\mathbb{A}_R^{d,e}}$ and $x\in \partial B(\Cr{c:M} R)$}.
\end{split}
\end{equation}
\end{Lemme}

\begin{proof}
Let $y\in{\mathbb{A}_R^{d,e}\cap G}$, then by definition, see \eqref{eq:defoverlineball} and \eqref{eq:defann}, there is a nearest-neighbor path $\pi=(x_i)_{i\in{\mathbb{N}}}\subset G$ such that $x_1=y$, $x_i\notin{B((1-e)R)}$ for all $i\geq1$, and $x_i\rightarrow\infty$ as $i\rightarrow\infty$. Applying the identity \eqref{eq:hittingcap} in a manner similar to \cite[(2.17)]{prevost2023passage} for instance, one finds that there exist constants  $\Cl[c]{c:hitting}\in{(0,1/24)}$  and $\Cl[c]{c:hitting2}>0$ such that 
\begin{equation}
\label{eq:hittingprobaimproved}
P_{u}\big(H_{B(v,\Cr{c:hitting}aR)}<H_{B(v,aR/8)^{\mathsf c}}\big)\geq \Cr{c:hitting2}\text{ for all }u,v\in{G}\text{ with }d(u,v)\leq 3\Cr{c:hitting}aR.
\end{equation}
Let $y_1=y$ and for each $k\geq1$, define recursively $y_{k+1}$ as the first vertex in $\Lambda(\Cr{c:hitting}aR)$ such that $B(y_{k+1},\Cr{c:hitting}aR)$ is visited by $\pi$ after last exiting $B(y_k,\Cr{c:hitting}aR)$. We denote by $p$ the smallest integer $q \geq 1$ such that $y_q\in{B((M+1)R)}$, for some constant $M\geq1$ that we will fix later. 
By \eqref{eq:defLambda} and \cite[(2.8)]{DrePreRod2}, we have $d(u,y_{k+1})\leq 3\Cr{c:hitting}aR$ for all $u\in{B(y_k,\Cr{c:hitting}aR)}$ and any $1\leq k\leq p-1$ as long as $aR\geq c$. In particular, for each $1\leq k\leq p-1$, noting that $B(y_{k+1},aR/8)\subset (\mathbb{A}_R^{a/2})^{\mathsf c}$ since $e\leq a/4$, it follows from \eqref{eq:hittingprobaimproved} that the diffusion starting in any point $u\in{B(y_k,\Cr{c:hitting}aR)}$ will reach $B(y_{k+1},\Cr{c:hitting}aR)$ before hitting $\mathbb{A}_R^{a/2}$ with probability at least $\Cr{c:hitting2}$. Noting that $y_k\in{B((M+1)R)}$ for all $1\leq k\leq p$ implies $p\leq C(M/a)^{\alpha}$ by \eqref{eq:defLambda}, using a chaining argument, we deduce that for all $y\in{\mathbb{A}_R^{d,e}\cap G}$,
\begin{equation*}
P_y\big(H_{\partial B(MR)}< H_{\mathbb{A}_R^{a/2}}\big)\geq \exp(-C(M/a)^{\alpha}).
\end{equation*}
Let $x\in{\partial B(MR)}$, then by the strong Markov property at time $H_{\partial B(MR)}$ we have
\begin{equation*}
P_y\big(H_x<H_{\mathbb{A}_R^{a/2}}\big)\geq \exp(-C(M/a)^{\alpha})\inf_{u\in{\partial B(MR)}}P_{u}\big(H_x<H_{\mathbb{A}_R^{a/2}}\big).
\end{equation*}
The last probability is bounded from below by
\begin{equation*}
\begin{split} 
P_{u}(H_x<\infty)-P_{u}\big(H_{\mathbb{A}_R^{a/2}}<\infty\big)\sup_{v\in{\mathbb{A}_R^{a/2}}}P_{v}(H_x<\infty)\geq  cM^{-\nu}R^{-\nu}-CM^{-2\nu}R^{-\nu},
\end{split}
\end{equation*}
where the last inequality follows from \eqref{eq:intro_Green}, \eqref{eq:hittingcap} and \cite[(2.8)]{DrePreRod2}. Fixing $M=\Cr{c:M}$ for a large enough constant $\Cr{c:M}$, the last two equations yield that for $x$ as above, \eqref{eq:hittingprobaxz} holds.
\end{proof}

The proof of Lemma~\ref{lem:newformulacap} utilizes the above result in combination with Proposition~\ref{P:lupu-gen}.

 \begin{proof}[Proof of Lemma~\ref{lem:newformulacap}]  Let $x \in \partial B(\Cr{c:M}R)$. 
Denoting by $\mathrm{cap}_a(K')$ the capacity of a set $K' \subset \tilde{\G}$ for the diffusion on $\tilde{\G}$ killed on $\mathbb{A}_R^{a/2}$, we have by the definition of the capacity in \eqref{eq:defcap} that $\mathrm{cap}_a(K')\geq \mathrm{cap}(K')$. Hence,   if $K'\subset \mathbb{A}_R^{d,e}$ has finitely many connected components, is compact, and $\mathrm{cap}(K')\geq t ((e-d)R)^{\nu}$, we have by \eqref{eq:hittingcap} applied on the graph with infinite killing on $\mathbb{A}_R^{a/2}$ that, with $P_x=P_x^{\tilde{\G}}$,
\begin{equation}
\label{eq:hittingprobaK'}
\begin{split} 
P_x\big(H_{K'}<H_{\mathbb{A}_R^{a/2}}\big)\geq c\cdot\mathrm{cap}(K')\inf_{z\in{K'}}P_x\big(H_{z}<H_{\mathbb{A}_R^{a/2}}\big)\geq cd^{\nu}\exp(-Ca^{-\alpha})t,
\end{split}
\end{equation}
where the last inequality follows from \eqref{eq:hittingprobaxz}, invariance by time reversal, and the inequality $d\leq e-d$. Recall now that $g^{\tilde{\G}_K}$ denotes the Green's function on the metric graph $\tilde{\G}_K$ introduced above Lemma~\ref{lem:newformulacap}. On the event $\mathrm{cap}(\mathcal{K}^0\cap \mathbb{A}_{R}^{d,e})\geq t((e-d)R)^{\nu}$, which implies in particular that $\mathcal{K}^0 \neq \emptyset$, by \eqref{eq:alt-green-diff} applied to $\tilde{\G}_K$, one finds that
\begin{equation*}
\begin{split}
    g_{\{0\}}^{\tilde{\G}_K}(x)-g_{\K^0}^{\tilde{\G}_K}(x)&=E_x\big[g_{\{0\}}^{\tilde{\G}_K}(X_{H_{\K^0}},x)1\{H_{\K^0}<H_{K\cup\{0\}}\}\big]
    \\&\geq cE_x\Big[P_{X_{H_{\K^0\cap \mathbb{A}_{R}^{d,e}}}}\big(H_x<H_{\mathbb{A}_R^{a/2}}\big)1\{H_{\K^0\cap \mathbb{A}_{R}^{d,e}}<H_{\mathbb{A}_R^{a/2}}\}\Big]
    \\&\geq cd^{\nu}\exp(-2Ca^{-\alpha})tR^{-\nu},
\end{split}
\end{equation*}
where we combined \eqref{eq:hittingprobaxz} and \eqref{eq:hittingprobaK'} in the last inequality.
The inequality \eqref{eq:newformulacap} now follows from \eqref{eq:lupu-gen}, applied on the graph $\G_K$, and the inequalities $\arctan(t)\leq t$ valid for all $t \geq 0$, $g^{\tilde{\G}_K}(0,x)=P_0^{\tilde{\G}_K}(H_x<\infty) g^{\tilde{\G}_K}(x)\leq CP_0^{\tilde{\G}}(H_x<H_K)$ and $g^{\tilde{\G}_K}(x),g^{\tilde{\G}_K}(0)\geq c$ by \eqref{eq:intro_Green} and \cite[(2.10)]{DrePreRod2}. To apply \eqref{eq:lupu-gen}, note that $\G_K$ satisfies the first condition in \eqref{T1_signsat}, and hence \eqref{eq:critpar0} as well. The former is true since the Green's function killed on $K$ is smaller than the Green function on $\mathcal{G}$ by definition, see \eqref{eq:defcap}, which one combines with  \eqref{eq:intro_Green} and \cite[Lemma~3.4,(2)]{DrePreRod3}.
\end{proof}

When $K=\emptyset$, \eqref{eq:newformulacap} can be seen in view of \eqref{eq:intro_Green} as a direct consequence of \eqref{eq:cap-tail}, since the event on the left-hand side of \eqref{eq:newformulacap} implies $\mathrm{cap}(\K^0)\geq s(dR)^{\nu}$. However, when $K$ is large, $P_0(H_x<H_K)$ can decrease significantly faster than $R^{-\nu}$ for $x\in{\partial B(\Cr{c:M}R)}$, see \cite[Lemma~2.1]{DrePreRod8}, and the formula \eqref{eq:newformulacap} becomes in a sense stronger than \eqref{eq:cap-tail}. We refer to \eqref{eq:proofmainprop1} and below as to where this improvement is needed, and combining this with ideas from \cite[Section~4]{DrePreRod8} we obtain the following result.

\begin{Lemme}
There exists $\Cl[c]{c:eps}>0$ such that for all $R\geq1$, $cR^{-1}\leq 2a<b<1$, $2d<e\leq a/4$ and $s,t>0$  
\begin{multline}
\label{eq:consequencenewcapformula}
\P\Big(\mathrm{cap}(\mathcal{K}^0\cap \mathbb{A}_{R}^{a,b})\leq s((b-a)R)^{\nu},\,\mathrm{cap}(\mathcal{K}^0\cap \mathbb{A}_{R}^{d,e})\geq t((e-d)R)^{\nu}\Big) 
\\ \leq  CR^{-\frac\nu2}t^{-\frac12}d^{-\frac\nu2}\exp\big(Ca^{-\alpha}-\Cr{c:eps}s^{-\frac1\nu}\big).
\end{multline}

\end{Lemme}
\begin{proof}
We use the isomorphism \cite{MR3417508,MR3502602} with the loop soup $\mathcal{L}$ on $\tilde{\G}$ at intensity $1/2$ on the metric graph $\tilde{\G}$, which we now review. We refer to \cite{MR3502602} for a detailed construction, and only recall here that $\mathcal{L}$ is a Poisson point process of Markovian loops on $\tilde{\G}$, that is defined under an auxiliary probability $\mathbb{Q}$. Moreover, if $\mathcal{C}$ is defined as the empty set with probability $1/2$, or otherwise denotes the cluster of $0$ in $\mathcal{L}$, that is the set of points in $\tilde{\G}$ which are connected to $0$ using a finite number of loops in $\mathcal{L}$, then
\begin{equation}
\label{eq:loopiso}
\begin{split} 
\mathcal{C}\stackrel{\text{law}}=\K^0.
\end{split}
\end{equation}
We will work under $\Q$ using the identification \eqref{eq:loopiso} throughout the proof.
The isomorphism \eqref{eq:loopiso} holds not only on ${\G}$ under the conditions listed below \eqref{eq:ellipticity}, but actually for any weighted transient graph even with a positive killing measure, and in particular on the graph $\G_K$ introduced above Lemma~\ref{lem:newformulacap}, for any finite set $K\subset G$. 

Let $\mathcal{L}^{\text{big}}\subset \mathcal{L}$ be obtained from $\mathcal{L}=\sum_i \delta_{\gamma_i}$ by retaining only \textit{big} (macroscopic) loops in the annulus $\mathbb{A}^{a,b}_R$, i.e.~loops $\gamma_i$ whose range satisfies  $\text{cap}(\text{range}(\gamma_i)) > s((b-a)R)^{\nu}$ and for which $\text{range}(\gamma_i) \subset \mathbb{A}^{a,b}_R$. Then on the event $\mathrm{cap}(\mathcal{C}\cap \mathbb{A}_{R}^{a,b})\leq s((b-a)R)^{\nu}$, and by the isomorphism \eqref{eq:loopiso}, $\K^0$ has the same law as the cluster of $0$ for the loop soup $\mathcal{L}\setminus \mathcal{L}^{\text{big}}$. Let us denote by $\mathcal{O}$ the intersection with $G$ of all the loops in $\mathcal{L}^{\text{big}}$. Using the restriction property for the loop soup  \cite[Theorem~6.1]{MR3238780} and the isomorphism \eqref{eq:loopiso} on the graph $\G_\mathcal{O}$, noting also that $\mathcal{L}^{\text{big}}$ and $\mathcal{L}\setminus\mathcal{L}^{\text{big}}$ are independent by defining properties of Poisson point processes, one deduces that 
\begin{multline}
\label{eq:proofviaiso}
    \mathbb{Q}\Big(\mathrm{cap}(\mathcal{C}\cap \mathbb{A}_{R}^{a,b})\leq s((b-a)R)^{\nu},\,\mathrm{cap}(\mathcal{C}\cap \mathbb{A}_{R}^{d,e})\geq t((e-d)R)^{\nu}\,\Big|\,\mathcal{L}^{\text{big}} \Big)
    \\ \leq \mathbb{P}^{\tilde{\G}_\mathcal{O}}\Big(\mathrm{cap}(\mathcal{K}^0\cap \mathbb{A}_{R}^{d,e})\geq t((e-d)R)^{\nu}\Big).
\end{multline}
We refer to \cite[(4.26)]{DrePreRod8} and above for a similar reasoning with more details. For a parameter $\delta>0$ to be fixed later, let us introduce
\begin{equation}
\label{eq:choiceLl}
\begin{split} 
    L\stackrel{\text{def.}}{=}s^{\frac1\nu}(b-a)R\delta^{-\frac1{\nu}}\text{ and }\ell\stackrel{\text{def.}}{=}\left\lfloor\frac{(b-a)R}{5L}\right\rfloor-3=\bigg\lfloor \frac15 \bigg( \frac{\delta}{s}\bigg)^{\frac1\nu}\bigg\rfloor-3,
\end{split}
\end{equation}
and note that the loops in $\mathcal{L}^{\text{big}}$ then have capacity at least $\delta L^{\nu}$ by definition, and that we can assume w.l.o.g.\ that  $\ell,L\geq1$ if $s\leq c=c(\delta)$ and $\delta\leq c'$, since otherwise \eqref{eq:consequencenewcapformula} is either trivial or follows easily from the capacity bounds \eqref{eq:cap-tail}. Let us now denote by $\mathbf{G}$ the event that $\mathcal{O}$ is a $(L,R,\ell/2,\delta L^{\nu})$-good obstacle set as defined above \cite[Lemma~2.1]{DrePreRod8}. That is, for any path $\pi$ in $\Lambda(L)$  from $0$ to $B(R)^c$, there is $A\subset \text{range}(\pi\cap B(R))$ such that $|A|\geq \ell/2$ and $\mathrm{cap}(B(y,L)\cap\mathcal{O})\geq \delta L^{\nu}$ for all $y\in{A}$. Combining \eqref{eq:proofviaiso} and \eqref{eq:newformulacap} for $K=\mathcal{O}$, we obtain that
\begin{align}
\begin{split}
\label{eq:probaonG}
&\mathbb{Q}\Big(\mathrm{cap}(\mathcal{C}\cap \mathbb{A}_{R}^{a,b})\leq s((b-a)R)^{\nu},\mathrm{cap}(\mathcal{C}\cap \mathbb{A}_{R}^{d,e})\geq t((e-d)R)^{\nu},\mathbf{G} \Big)
\\ 
&\leq \inf_{x\in{\partial B(\Cr{c:M}R)}} CR^{\frac\nu2}t^{-\frac12}d^{-\frac\nu2}\exp(Ca^{-\alpha})\mathbb{E}^{\mathbb{Q}}\big[P_0(H_x<H_\mathcal{O})1\{\mathbf{G}\}\big]
\\
&\leq  CR^{-\frac\nu2}t^{-\frac12}d^{-\frac\nu2}\exp(Ca^{-\alpha})\exp(-c\delta \ell),
\end{split}
\end{align}
where the last inequality follows from \cite[Lemma~2.1]{DrePreRod8}, up to assuming w.l.o.g.\ that $L\leq cR$, that is $s\leq c'$ for some small enough constant $c'=c'(\delta)$.

It remains to control the probability on the left-hand side of \eqref{eq:consequencenewcapformula} on the event $\mathbf{G}^c$, which relies on a reasoning similar to, but somewhat simpler than, \cite[Lemma~2.3]{DrePreRod8}.  
Let us denote by $\mathcal{P}$ the set of tuples $\tau=(x_1,\dots,x_\ell)$ such that such that $x_i\in{\Lambda(L)}$ and $B(x_i,L)\subset \mathbb{A}^{a,b}_R$ for all $1\leq i\leq \ell$, $B(x_i,L)\cap B(x_j,L)=\emptyset$ and $x_{i+1}\in{B(x_i,5L)}$ for all $1\leq i<j\leq \ell$. Here $\ell$ and $L$ are as defined in \eqref{eq:choiceLl}.
We write $D_{\tau}$ for the set of $i\in{\{1,\dots,\ell\}}$ such that there exists a loop in $\mathcal{L}^{\text{big}}$ whose range is included in $B(x_i,L)$.

Let us now show that any path $\pi$ in $\Lambda(L)$ from $0$ to $B(R)^{\mathsf c}$ contains a tuple $\tau\in{\mathcal{P}}$, in the sense that $\text{range}(\tau)\subset\text{range}(\pi)$. We call $\pi'$ the subpath of $\pi$ which starts just after $\pi$ last visiting $\overline{B}((1-b)R+L)$, and afterwards stops just before first leaving $B((1-a)R-L)$. Assuming w.l.o.g.\ that $(b-a)R\geq cL$, that is  $s\leq c'$,  one can easily check  by \cite[(2.8)]{DrePreRod2} that $\pi'$ is non-empty.  Let us now define recursively $x_1$ as the first vertex in $\pi'$, and recursively $x_{k+1}$ as the first vertex in $\pi'$ visited after last exiting $B(x_{k},2L)$, and denote by $p$ the smallest integer $k \geq 1$ such that $\pi'$ never exits $B(x_{k},2L)$. Note that by \cite[(2.8)]{DrePreRod2}, we have $d(x_k,x_{k+1})\leq 5L$ for all $1\leq k\leq \ell$ whenever $L\geq C$, and hence $d(x_1,x_p)\leq 5pL$. Since moreover $x_1\in{B((1-b)R+8L)}$ and $x_p\in{B((1-a)R-4L)^{\mathsf c}}$ for $L\geq C$, we deduce that $5pL\geq (b-a)R-12L$, and hence $p\geq \ell$ in view of \eqref{eq:choiceLl}. Noting additionally that $B(x_i,L)\subset\mathbb{A}_R^{a,b}$ for all $1\leq i\leq \ell$, we thus obtain all in all that  $\tau=(x_1,\dots,x_\ell)\in{\mathcal{P}}$.

By the previous paragraph and by definition of $\mathbf{G}$, there exists on the event $\mathbf{G}^c$ a tuple $\tau\in{\mathcal{P}}$ such that $|D_{\tau}|\leq\ell/2$. Moreover, for each $\tau=(x_1,\dots,x_\ell)\in{\mathcal{P}}$, the events $\{\exists\, \gamma_i\in{\mathcal{L}^{\text{big}}},\text{range}(\gamma_i)\subset B(x_i,L)\}$, $1\leq i\leq \ell$, are i.i.d.\ by properties of Poisson point process, and occur with probability at least $p=p(\delta)$ which satisfies $p(\delta)\rightarrow 1$ as $\delta\rightarrow0$ by \cite[(4.31)]{DrePreRod8}. Since by \eqref{eq:defLambda} we have $|\mathcal{P}|\leq C^{\ell}$, we deduce by a union bound that
\begin{equation}
\label{eq:probaGcomp}
    \mathbb{Q}\big(\mathbf{G}^{\mathsf c}\big)\leq C^{\ell}\sup_{\tau\in{\mathcal{P}}}\mathbb{Q}(D_{\tau})\leq C^{\ell}2^{\ell}(1-p)^{\ell/2}\leq \exp(-c\ell),
\end{equation}
where the last inequality holds when $\delta=c$ for a small enough constant $c>0$. Now, the tail asymptotic \eqref{eq:cap-tail} holds on the graph $\G_\mathcal{O}$ by application of \cite[Theorem~1.1]{DrePreRod3}, since the first condition in \eqref{T1_signsat}, and hence \eqref{eq:critpar0} as well, are satisfied on that graph by virtue of \cite[Lemma~3.4,(2)]{DrePreRod3}. It thus follows from \eqref{eq:cap-tail} on the graph $\G_\mathcal{O}$ and \eqref{eq:proofviaiso} that 
\begin{multline}
\label{eq:probaonGcomp}
    \mathbb{Q}\big(\mathrm{cap}(\mathcal{C}\cap \mathbb{A}_{R}^{a,b})\leq s((b-a)R)^{\nu},\mathrm{cap}(\mathcal{C}\cap \mathbb{A}_{R}^{d,e})\geq t((e-d)R)^{\nu},\mathbf{G}^{\mathsf c} \big)
    \\\leq \E^{\mathbb{Q}}\big[\mathbb{P}^{\tilde{\G}_\mathcal{O}}(\mathrm{cap}(\K^0)\geq t((e-d)R)^{\nu})1\{\mathbf{G}^c\}\big]\leq Ct^{-\frac12}(e-d)^{-\frac\nu2}R^{-\frac\nu2}\mathbb{Q}(\mathbf{G}^c).
\end{multline}
Combining \eqref{eq:loopiso}, \eqref{eq:choiceLl}, \eqref{eq:probaonG}, \eqref{eq:probaGcomp} and \eqref{eq:probaonGcomp} for $\delta=c$ as before, assuming w.l.o.g.\ that $s\leq c'$, and recalling that $e-d\geq d$, the claim follows. 
\end{proof} 

We are now ready to finish the:
\begin{proof}[Proof of Proposition~\ref{P:t_R-LB}]
Combining \eqref{eq:1-arm-UB1} with \eqref{eq:consequencenewcapformula}, we have, for all $s>0$ and $t\leq t_{R,\eps}^{d,e}$,
\begin{multline}
\label{eq:proofmainprop1}
\P\Big(\mathcal{K}^0 \cap \partial \overline{B}(R) \neq \emptyset, \, \mathrm{cap}(\mathcal{K}^0\cap \mathbb{A}_{R}^{a,b})\leq s((b-a)R)^{\nu}\Big)\\\leq R^{-\frac{\nu}{2}}\exp\big(-\eps t^{-\frac1\nu}\big)+ CR^{-\frac\nu2}t^{-\frac12}d^{-\frac\nu2}\exp(Ca^{-\alpha}-\Cr{c:eps}s^{-\frac1\nu}\big).
\end{multline}
We now take
\begin{equation*}
\begin{split} 
\eps=\Cr{c:eps2}\stackrel{\text{def.}}{=}\frac{\Cr{c:eps}}{4}\text{ and } t=\exp\big(-\Cr{c:eps}s^{-\frac1\nu}\big),
\end{split}
\end{equation*}
then if $s\leq c(a^{\alpha\nu}\wedge\log(1/d)^{-\nu})$ for some small enough constant $c>0$,  one can bound the right-hand side of \eqref{eq:proofmainprop1} from above by $R^{-\nu/2}\exp\big(-\eps s^{-\frac1\nu}\big)$. Noting that the condition $t\leq t_{R,\eps}^{d,e}$ for the above choice of $t$ is satisfied whenever $s\leq \Cr{c:eps}^{\nu} \log(t_{R,\eps}^{d,e})^{-\nu}$, we conclude in view of \eqref{eq:1-arm-UB1} that any $s$ satisfying all previous requirements is upper bounded by $t_{R}^{a,b}=t_{R,\Cr{c:eps2}}^{a,b}$, i.e.,~\eqref{eq:t_R-LB1} holds.
\end{proof}

\bibliography{bibliographie}
\bibliographystyle{abbrv}

\end{document}